\documentclass[a4paper,12pt]{article}
\usepackage[top=2.5cm,bottom=2.5cm,left=2.5cm,right=2.5cm]{geometry}
\usepackage{cite, amsmath, amssymb}
\usepackage{amssymb}
\usepackage{mathrsfs}
\usepackage{graphicx}
\newenvironment{proof}{{\noindent \it Proof.}}{\hfill $\blacksquare$\par}
\usepackage{latexsym}
\usepackage{graphicx,booktabs,multirow}
\usepackage{tikz}
\usetikzlibrary{decorations.pathreplacing}
\usetikzlibrary{intersections}
\usepackage{enumerate}
\usepackage{graphicx,booktabs,multirow}
\usepackage{appendix}
\newtheorem{theorem}{Theorem}[section]

\newtheorem{lemma}[theorem]{Lemma}
\newtheorem{corollary}[theorem]{\rm\bfseries Corollary}

\begin{document}

\title{On the first three minimum Mostar indices of tree-like phenylenes}
\author{Hechao Liu$^{1}$, Lihua You$^{1,}$\thanks{Corresponding author}, Hanlin Chen$^{2}$, Zikai Tang$^{3}$
 \\
{\small $^1$School of Mathematical Sciences, South China Normal University,}\\ {\small Guangzhou, 510631, P. R. China}\\
 \small {\tt hechaoliu@m.scnu.edu.cn},\quad  \small {\tt ylhua@scnu.edu.cn}\\
{\small $^2$College of Computer Engineering and Applied Mathematics, Changsha University,}\\ {\small Changsha, Hunan 410022, P. R. China}\\
\small {\tt hlchen@ccsu.edu.cn}\\
{\small $^3$School of Mathematics and Statistics, Hunan Normal University,}\\ {\small Changsha, Hunan 410081, P. R. China}\\
\small {\tt zikaitang@163.com}
}
\date{}
\maketitle
\begin{abstract}
Let $G =(V_{G}, E_{G})$ be a simple connected graph with its vertex set $V_{G}$ and edge set $E_{G}$.
The Mostar index $Mo(G)$ was defined as $Mo(G)=\sum\limits_{e=uv\in E(G)}|n_{u}-n_{v}|$, where $n_{u}$ (resp., $n_{v}$) is the number of vertices whose distance to vertex $u$ (resp., $v$) is smaller than the distance to vertex $v$ (resp., $u$). In this study, we determine the first three minimum Mostar indices of tree-like phenylenes and characterize all the tree-like phenylenes attaining these values. At last, we give some numerical examples and discussion.
\end{abstract}

\maketitle

\makeatletter
\renewcommand\@makefnmark%
{\mbox{\textsuperscript{\normalfont\@thefnmark)}}}
\makeatother

 \baselineskip=0.30in

\section{Introduction}
\hskip 0.6cm
Chemical indices are a class of numerical invariants that are closely related to the structure of a chemical graph. Chemical indices can be used to predicte the structural and physic-chemical properties of a chemical molecular. They are mainly used for the quantitative characterisation of chemical structures and thus contribute to the study of QSAR, QSTR and QSPR relationships of chemical structures. In recent years, chemical indices have been found a wide range of applications in chemical science, medical science, complex networks, toxicology, etc. The combination of
quantum chemistry and chemical graph theory has also become a promising area in the study of QSAR and QSPR.

Our convention in this paper follows \cite{frkc1997} for notations that we omit here.
Phenylenes are a class of chemical compounds in which the carbon atoms form 6-membered cycles and 4-membered cycles. Each 4-membered cycle(=square) is adjacent to two disjoint 6-membered cycles(=hexagons), and no two hexagons are adjacent \cite{pagu1997}. If there is a hexagon of phenylene adjacent to three squares, we call it the tree-type phenylene (the definition of tree-type phenylene is similar to tree-type hexagonal system, see \cite{dngl2020}). Related structures include extended sp-carbon nets and heterocyclic analogs. These molecules have great theoretical and potential practical significance in finding new molecules with (super) conductive properties \cite{kpcv1993}.

Nowadays, phenylene is still a hot topic in many experimental and theoretical studies.
Some topological properties of phenylenes has been established such as (total $\pi$-electron) energy \cite{gufa2017}, HOMO LUMO separation \cite{gupk1996}, cyclic conjugation \cite{gtto2001}, Kekul\'{e} structure count\cite{gtfu2006}, Wiener index \cite{fgzv2002}, PI index \cite{dcz2007,guar2008}, Detour index \cite{lfan2021} and Kirchhoff index \cite{lzch2020,zzli2019}.

Let $\mathscr{P}_h$ be the set of tree-like phenylenes with $h$ hexagons and $h-1$ squares. And $\mathscr{P}=\bigcup_{h=1}^{\infty}\mathscr{P}_h$.
Suppose $P\in \mathscr{P}_h$, and $R$ is a hexagon of $P$. Then $R$ is called an $i$-hexagon, if it has exactly $i$ $(0\leq i \leq 3)$ adjacent squares in $P$. A $1$-hexagon is called a terminal hexagon of $P$. A $2$-hexagon is called turn-hexagon of $P$ if its two $2$-vertices (the vertices of degree $2$) are adjacent in $P$. A $3$-hexagon is called a full-hexagon of $P$. Let $\mathscr{P}_{h,i}\subseteq \mathscr{P}_h$ $(0\leq i\leq \lfloor\frac{h-1}{2}\rfloor)$ be the set of phenylenes with $i$ full-hexagons.
Each graph in $\mathscr{P}_{h,0}$ (or denoted by $\mathcal{C}_{h}$ ) is called a phenylene chain. A phenylene chain is called a linear phenylene chain (denoted by $L_{h}$) if it contains no turn-hexagons. Let $\mathcal{C}_{h,i}\subseteq \mathcal{C}_{h}$ be the set of  phenylene chains with $i$ turn-hexagons.

A segment of a phenylene chain $G\in \mathcal{C}_{h}$ is a maximal linear sub-chain. Denote by $S$ a non-terminal segment of a phenylene chain $G$.
We call $S$ a non-zigzag segment (resp., a zigzag segment) if it's two neighboring segments lie on the same sides (resp., on different sides) of the line through centers of all hexagons and squares on $S$.

Denote by $C_{L}(t_{1},t_{2},t_{3},\cdots, t_{k}, t_{k+1})$ the phenylene chain with $h$ hexagons and exactly $k+1$ segments $S_1, S_2,\cdots, $ $S_{k+1}$ of lengths $t_{1}+1,t_{2}+2,t_{3}+2,\cdots,t_{k}+2, t_{k+1}+1$, respectively, where $S_1$ and $S_{k+1}$ are the terminal segments, all $S_i$ $(2\leq i\leq k)$ are zagzig segments, $1\leq t_1\leq t_{k+1}$, and $\sum\limits_{i=1}^{k+1}t_{i}+k=h$.
Particularly, $C_{L}(j,n)\in \mathcal{C}_{h,1}$ is the graph including two vertex-disjoint linear phenylene chains $L_{j}$ and $L_{n}$ as subgraphs, where $j\leq n$, and $h=j+n+1$. $C_{L}(j,k,n)\in \mathcal{C}_{h,2}$ is the graph including three vertex-disjoint linear phenylene chains $L_{j}$, $L_{k}$ and $L_{n}$ as subgraphs, where $j\leq n$, the second segment is a zigzag segment and $h=j+k+n+2$.

Denote by $P_{L}(j,k,n)\in \mathscr{P}_{h,1}$ the graph including three vertex-disjoint linear phenylene chains $L_{j}$, $L_{k}$ and $L_{n}$ as subgraphs, where $j\leq k\leq n$, and $h=j+k+n+1$.

Do\v{s}li\'{c} et al.\cite{dmstz2018} introduced Mostar index \cite{dmstz2018} of a graph $G$, which is defined as
$$Mo(G)=\sum_{e=uv\in E(G)}|n_{u}-n_{v}|.$$
Do\v{s}li\'{c} et al. determined extremal values of Mostar index among trees and unicyclic graphs, then gave a cut method for computing the Mostar index of benzenoid systems.

Similarly, the edge Mostar index \cite{iari2020} is defined as
$$Mo_{e}(G)=\sum_{e=uv\in E(G)}|m_{u}-m_{v}|,$$
where $m_{u}$ (resp.,$m_{v}$) is the number of edges whose distance to vertex $u$ (resp., $v$) is smaller than the distance to vertex $v$ (resp., $u$). We can refer to \cite{aocm2019,mjnk2020,kbla2020,dngl2020,dgli2021,dmstz2018,gaxu2020,hlzh2020,hyzu2019,hyuz2019,iari2020,lsxt2020,
teph2019,tank2020,xztd2021,xzdt2021} for more details about (edge) Mostar index.

Mostar indices can be used to measure the peripherality of chemical graphs, so in addition to chemical applications, Mostar indices have a wide range of applications in complex networks. It can be used to describe structural properties of the network. Mostar indices can also be used to extend quantum estimates and expand reactivity based on electronic descriptors.
In this study, our aim is to solve the extremal problem of tree-like phenylenes with respect to Mostar indices. Using the methods of \cite{dngl2020}, we determine the first three minimum values of the Mostar index of tree-like phenylenes with a fixed number of hexagons and characterize all the tree-like phenylenes attaining these values.

\section{Preliminaries}
\hskip 0.6cm
An orthogonal cut is a line segment that starts from the middle of a peripheral edge of a phenylene, goes orthogonal to this edge and ends at the first next peripheral edge that it intersects. Let $P\in \mathscr{P}_h$. Denote by $\mathcal{O}_{P}^{uv}$ the set of edges that parallel with $uv$ in $P$, and $|\mathcal{O}_{P}^{uv}|=o_{P}^{uv}$. Let $\mathcal{O}_{P}$ be the set of all disjoint parallel classes in $P$.
Note that $\mathcal{O}_{P}^{uv}$ is an edge cut, denote by $G_{P}^{u}$ $(resp., G_{P}^{v})$ the connected components of $P-\mathcal{O}_{P}^{uv}$ contain $u$ $(resp., v)$. Denote by $r_{P}^{u}$ $(resp., r_{P}^{v})$ the number of hexagons in $G_{P}^{u}$ $(resp., G_{P}^{v})$. Note that, for any
$P\in \mathscr{P}_h$, $|V_{P}|=6h$ and $|E_{P}|=8h-2$.

Bearing in mind that $P_{L}(j,k,n)\in \mathscr{P}_{h,1}$ is the graph including three vertex-disjoint linear phenyene chains $L_{j}$, $L_{k}$ and $L_{n}$ as subgraphs, where $j\leq k\leq n$, and $h=j+k+n+1$.
By the definition of $Mo(G)$, we can calculate the value of $Mo(P_{L}(j,k,n))$.
\begin{lemma}\label{l2-3}
Given a phenylene $G=P_{L}(j,k,n)\in \mathscr{P}_{h,1}$ with three branches $L_{j}, L_{k}, L_{n}$ $(1\leq j\leq k\leq n)$ and $h$ $(h=j+k+n+1)$ hexagons, then

$(1)$ If $n\leq \lfloor\frac{h}{2}\rfloor$, then $Mo(G)=24(2kj+3nj+4kn+k+2n)$.

$(2)$ If $n\geq \lfloor\frac{h}{2}\rfloor+1$, then $Mo(G)=6(4j+3j^{2}+8k+3k^{2}+4n+3n^{2}+14kj+6jn+10kn+1)$ for even $h$; $Mo(G)=6(4j+3j^{2}+8k+3k^{2}+4n+3n^{2}+14kj+6jn+10kn)$ for odd $h$.
\end{lemma}
\begin{proof}
By using the cut method to $P_{L}(j,k,n)$, we have
\begin{align*}
 Mo(G)=&6\{2(j+1)(n-k)+2(k+1)(n-j)+2(n+1)(k-j)\\
 &+4\sum_{i=1}^{j}(h+1-2i)+4\sum_{i=1}^{k}(h+1-2i)+4\sum_{i=1}^{n}|h+1-2i|\\
 &+2\sum_{i=1}^{j}(h-2i)+2\sum_{i=1}^{k}(h-2i)+2\sum_{i=1}^{n}|h-2i|\}.
\end{align*}

$(1)$ If $n\leq \lfloor\frac{h}{2}\rfloor$, then
\begin{align*}
 Mo(G)=&6\{2(j+1)(n-k)+2(k+1)(n-j)+2(n+1)(k-j)\\
 &+4\sum_{i=1}^{j}(h+1-2i)+4\sum_{i=1}^{k}(h+1-2i)+4\sum_{i=1}^{n}(h+1-2i)\\
 &+2\sum_{i=1}^{j}(h-2i)+2\sum_{i=1}^{k}(h-2i)+2\sum_{i=1}^{n}(h-2i)\}\\
 =&24(2kj+3nj+4kn+k+2n).
\end{align*}

$(2)$ If $n\geq \lfloor\frac{h}{2}\rfloor+1$, and $h$ is even, then
\begin{align*}
 Mo(G)=&6\{2(j+1)(n-k)+2(k+1)(n-j)+2(n+1)(k-j)\\
 &+4\sum_{i=1}^{j}(h+1-2i)+4\sum_{i=1}^{k}(h+1-2i)+4\sum_{i=1}^{\frac{h}{2}}(h+1-2i)-4\sum_{i=\frac{h}{2}+1}^{n}(h+1-2i)\\
 &+2\sum_{i=1}^{j}(h-2i)+2\sum_{i=1}^{k}(h-2i)+2\sum_{i=1}^{\frac{h}{2}}(h-2i)-2\sum_{i=\frac{h}{2}+1}^{n}(h-2i)\}\\
 =&6(4j+3j^{2}+8k+3k^{2}+4n+3n^{2}+14kj+6jn+10kn+1).
\end{align*}

If $n\geq \lfloor\frac{h}{2}\rfloor+1$, and $h$ is odd, similarly, we have $Mo(G)=6(4j+3j^{2}+8k+3k^{2}+4n+3n^{2}+14kj+6jn+10kn)$.

This completes the proof.
\end{proof}

\section{The minimal tree-like phenylenes}
\hskip 0.6cm
Let $G\in \mathcal{C}_{h}$ and $R_{1}$, $R_{h}$ are two terminal hexagons of $G$. Denote by $x_{i,1}, x_{i,2}, \cdots, x_{i,6}$ the six clockwise successive vertices in $R_{i}$ for $i=1,h$, where $d_{G}(x_{i,j})=2$ for $j=1,2,3,4$.
Let $e_{ij}=x_{i,j}x_{i,j+1}$ for $i=1,h$ and $j=1,2,\cdots, 6$ (let $x_{i,7}:=x_{i,1}$).

Suppose that $P_{1},P_{2}\in \mathscr{P}$ and $u_{i},v_{i}$ be two adjacent $2$-vertices (vertices with degree 2) in $P_{i}$ for $i=1,2$. Let $P=P_{1}(u_{1},v_{1}) \Box  P_{2}(u_{2},v_{2})$ be the phenylene obtained from $P_{1}, P_{2}$ by connecting $u_{1}$ with $u_{2}$, and $v_{1}$ with $v_{2}$, respectively.

\begin{figure}[ht!]
  \centering
  \scalebox{.15}[.15]{\includegraphics{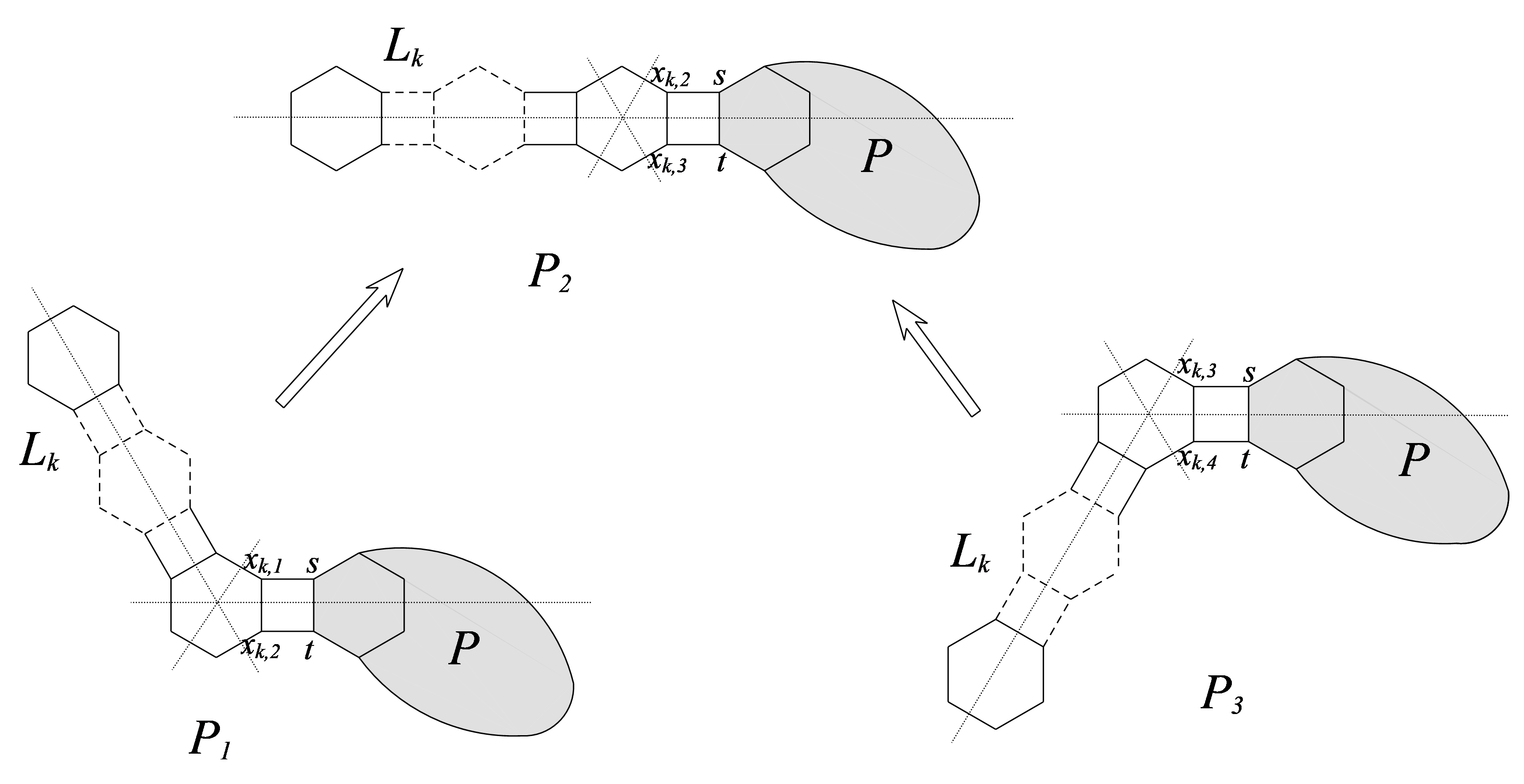}}
  \caption{The tree-like phenylenes $P_{1}$, $P_{2}$ and $P_{3}$ of Lemma \ref{l3-1}.}
 \label{fig-1}
\end{figure}

For $uv\in E(G)$, we denote $n_{G}^{u}$ (or $n_{u}$) the number of vertices in $G$ lying closer to vertex $u$ than to vertex $v$. Some symbols involved in the proof are described in Section 2.
\begin{lemma}\label{l3-1}
Let $P\in \mathscr{P}_{n}$ and $P_{i}=P(s,t) \Box  L_{k}(x_{k,i},x_{k,i+1})$ for $i=1,2,3$, see Figure \ref{fig-1}. And $k\geq 2$, $n\geq k-1$, then

$(1)$ $Mo(P_{2})\leq Mo(P_{1})$, with equality iff $|r_{P}^{t}-r_{P}^{s}|\geq k-1$ and $\min \{r_{P}^{s}, r_{P}^{t}\}=0$.

$(2)$ $Mo(P_{2})< Mo(P_{3})$.
\end{lemma}
\begin{proof}
Let $\mathcal{O}_{P}^{st}$ denote the set of edges that are parallel to edge $st$ in $P$.
Assume that $r_{P}^{s}\leq r_{P}^{t}$. Denote $E_{d}=\{e_{k1}\}\cup \{e_{k3}\}\cup \{e_{k4}\}\cup \{e_{k6}\} \cup \mathcal{O}_{P}^{st} \cup \mathcal{O}_{L_{k}}^{e_{k2}}$.
Denote $\phi_{i}(E_{d})=\sum\limits_{uv\in E_{d}}|n_{P_{i}}^{u}-n_{P_{i}}^{v}|$ ($i=1,2,3$), we have $Mo(P_{i})-Mo(P_{2})=\phi_{i}(E_{d})-\phi_{2}(E_{d})$ ($i=1,3$).

Note that $n\geq k-1$, then

$ \phi_{1}(E_{d})=6\{2(n-k+1)+2kn+(o_{P}^{st}+2)|(r_{P}^{t}-r_{P}^{s})-(k-1)|\}$;

$ \phi_{2}(E_{d})=6\{4(n-k+1)+(o_{P}^{st}+2k)(r_{P}^{t}-r_{P}^{s})\}$;

$ \phi_{3}(E_{d})=6\{2(n-k+1)+2kn+(o_{P}^{st}+2)[(r_{P}^{t}-r_{P}^{s})+(k-1)]\}$.

Bearing in mind that $k\geq 2$ and $n=\frac{1}{2}o_{P}^{st}+r_{P}^{t}+r_{P}^{s}$.

\noindent $(1)$ If $r_{P}^{t}-r_{P}^{s}\geq k-1$, then
$\phi_{1}(E_{d})-\phi_{2}(E_{d})=6(k-1)(2n-o_{P}^{st}-2(r_{P}^{t}-r_{P}^{s}))=24(k-1)r_{P}^{s}\geq 0$, with equality iff $r_{P}^{s}=0$.

If $r_{P}^{t}-r_{P}^{s}< k-1$, then
$\phi_{1}(E_{d})-\phi_{2}(E_{d})=12(o_{P}^{st}+2)(k-1-(r_{P}^{t}-r_{P}^{s}))>0$.

\noindent $(2)$ $\phi_{3}(E_{d})-\phi_{2}(E_{d})=6(k-1)(2n+4+o_{P}^{st}-2(r_{P}^{t}-r_{P}^{s}))=12(k-1)(o_{P}^{st}+2r_{P}^{s}+2)> 0$.

This completes the proof.
\end{proof}
\ \notag\

By Lemma \ref{l3-1}, we have
\begin{corollary}\label{c3-2}
Let $P\in \mathscr{P}_{n}$  with $s,t$ being two adjacent $2$-vertices of its turn hexagon. Denote by $C_{k}$ a phenylene chain with $k$ $( \geq 1 ) $ hexagons, and $u,v$ are two adjacent $2$-vertices of its one terminal hexagon. Then
$Mo(P(s,t) \Box L_{k}(x_{k,2},x_{k,3})) \leq Mo(P(s,t) \Box  C_{k}(u,v))$.
\end{corollary}

Let $P\in \mathscr{P}_{n+1}$, where $n\geq max\{j,k\}$, $j,k\geq 1$. $t_{1},t,s,s_{1}$ are vertices in Figure \ref{fig-2}, and $r_{P}^{s}\leq r_{P}^{t}$, $h=j+k+n+1$. Then we have

\begin{figure}[ht!]
  \centering
  \scalebox{.08}[.08]{\includegraphics{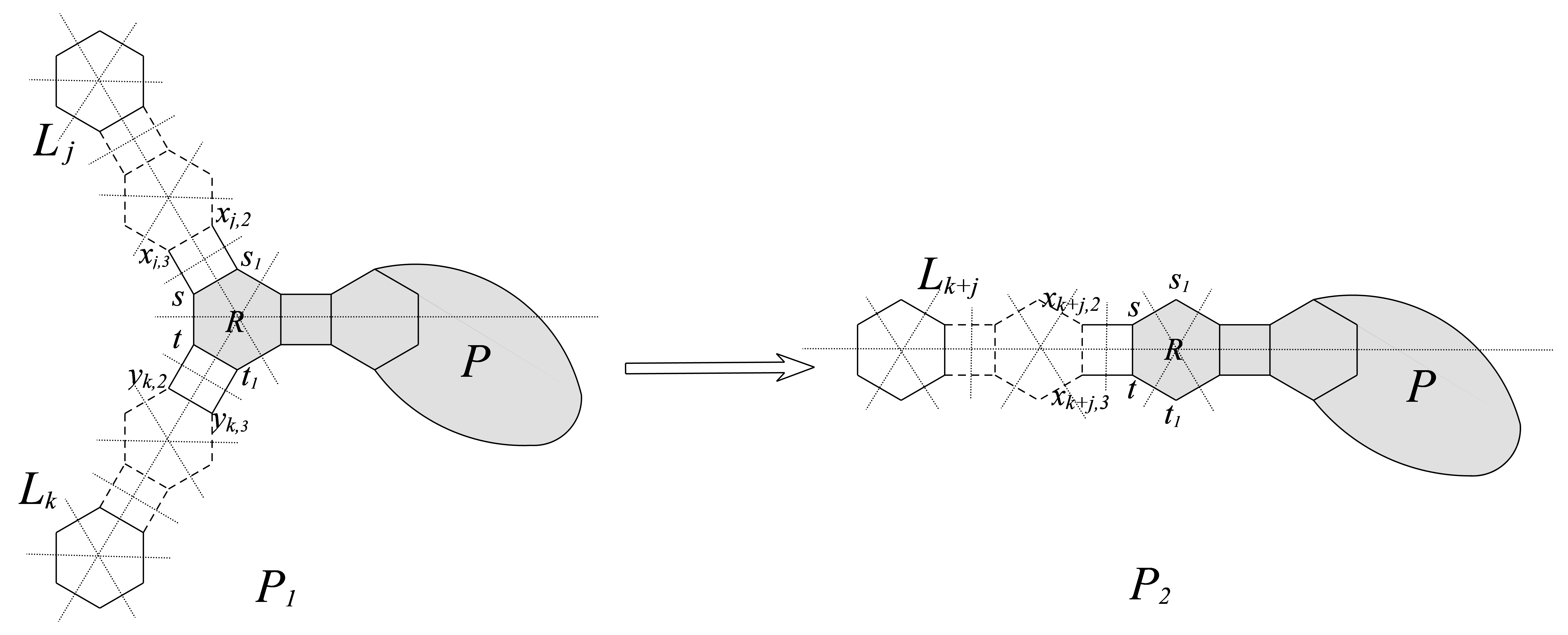}}
  \caption{The tree-type phenylenes $P_{1}$, $P_{2}$ of Lemma \ref{l3-3}.}
 \label{fig-2}
\end{figure}

\begin{lemma}\label{l3-3}
Let $P_{1}=\{P(t,t_{1})\Box L_{k}(y_{k,2},y_{k,3})\}(s_{1},s)\Box L_{j}(x_{j,2},x_{j,3})$ and $P_{2}=P(s,t)\Box L_{k+j}$ $(x_{k+j,2},x_{k+j,3})$, see Figure \ref{fig-2}, then $Mo(P_{2})<Mo(P_{1})$.
\end{lemma}
\begin{proof}
Denote $E_{d,i}=(E_{P_{i}}-E_{P})\cup E_{R}\cup \mathcal{O}_{P-R}^{st}$ ($i=1,2$).
Let $\phi_{i}(E_{d,i})=\sum\limits_{uv\in E_{d,i}}|n_{P_{i}}^{u}-n_{P_{i}}^{v}|$ ($i=1,2$), then $Mo(P_{1})-Mo(P_{2})=\phi_{1}(E_{d,1})-\phi_{2}(E_{d,2})$.

Since $n\geq \max\{j,k\}$, then
\begin{align*}
\phi_{1}(E_{d,1})=&6\{2(j+1)(n-k)+2(k+1)(n-j)+o_{P}^{st}|r_{P}^{t}+k-(r_{P}^{s}+j)|\\
 &+4\sum_{i=1}^{j}(h+1-2i)+4\sum_{i=1}^{k}(h+1-2i)+2\sum_{i=1}^{j}(h-2i)+2\sum_{i=1}^{k}(h-2i)\}.
\end{align*}
\begin{align*}
\phi_{2}(E_{d,2})=&6\{(o_{P}^{st}+2(j+k))(r_{P}^{t}-r_{P}^{s})+4\sum_{i=1}^{j+k+1}|h+1-2i|+2\sum_{i=1}^{j+k}|h-2i|\}.
\end{align*}

Note that $n=\frac{1}{2}o_{P}^{st}+r_{P}^{t}+r_{P}^{s}-1$, $j,k\geq 1$ and $o_{P}^{st}\geq 2$.

\noindent {\bf Case 1}. $r_{P}^{t}+k\geq r_{P}^{s}+j$.

\noindent {\bf Subcase 1.1}. $j+k+1\leq \lfloor\frac{h}{2}\rfloor$.

$\phi_{1}(E_{d,1})-\phi_{2}(E_{d,2})=12(o_{P}^{st}k+2(k+j)r_{P}^{s}+4kj)>0$.

\noindent {\bf Subcase 1.2}. $j+k+1\geq \lfloor\frac{h}{2}\rfloor+1$.

If $h$ is even, then $j+k\geq n+1$, and
\begin{align*}
\phi_{1}(E_{d,1})-\phi_{2}(E_{d,2})
=&6\{2o_{P}^{st}k+4(k+j)r_{P}^{s}-3j^{2}-3k^{2}-3n^{2}+2kj+6jn+6kn-4k\\
 & -4j+4n-1\}\\
=&6\{3k(n-k)+3(j+1)(n-j)+2k(o_{P}^{st}+r_{P}^{s}-2)+2j(k+2r_{P}^{s}-1)\\
 &+3n(j+k-n-1)+4(n-1)+2kr_{P}^{s}+j+3\}>0.
\end{align*}

If $h$ is odd, then $j+k\geq n$, and
\begin{align*}
\phi_{1}(E_{d,1})-\phi_{2}(E_{d,2})
=&6\{2o_{P}^{st}k+4(k+j)r_{P}^{s}-3j^{2}-3k^{2}-3n^{2}+2kj+6jn+6kn-4k\\
& -4j+4n\}\\
=&6\{3k(n-k)+3(j+1)(n-j)+2k(o_{P}^{st}+r_{P}^{s}-2)+2j(k+2r_{P}^{s}-1)\\
 &+3n(j+k-n)+n+2kr_{P}^{s}+j\}>0.
\end{align*}

\noindent {\bf Case 2}. $r_{P}^{t}+k< r_{P}^{s}+j$.

Note that $n\geq j\geq k+1\geq 2$ and $o_{P}^{st}\geq 2$.

\noindent {\bf Subcase 2.1}. $j+k+1\leq \lfloor\frac{h}{2}\rfloor$.

$\phi_{1}(E_{d,1})-\phi_{2}(E_{d,2})=12\{o_{P}^{st}(r_{P}^{s}+j-r_{P}^{t})+2(k+j)r_{P}^{s}+4kj\}>0$.

\noindent {\bf Subcase 2.2}. $j+k+1\geq \lfloor\frac{h}{2}\rfloor+1$.

If $h$ is even, then $j+k\geq n+1$, and
\begin{align*}
\phi_{1}(E_{d,1})-\phi_{2}(E_{d,2})
=&6\{2o_{P}^{st}(r_{P}^{s}+j-r_{P}^{t})+4(k+j)r_{P}^{s}-3j^{2}-3k^{2}-3n^{2}+2kj+6jn\\
& +6kn-4k-4j+4n-1\}\\
=&6\{3k(n-k)+3(j+1)(n-j)+3n(j+k-n-1)+2o_{P}^{st}(r_{P}^{s}+j\\
 &-r_{P}^{t}-k-1)+2k(j-2)+2(o_{P}^{st}-2)+(n-j)+4(k+j)r_{P}^{s}\\
 &+2o_{P}^{st}k+3n+3\}>0.
\end{align*}

If $h$ is odd, then $j+k\geq n$, and
\begin{align*}
\phi_{1}(E_{d,1})-\phi_{2}(E_{d,2})
=&6\{2o_{P}^{st}(r_{P}^{s}+j-r_{P}^{t})+4(k+j)r_{P}^{s}-3j^{2}-3k^{2}-3n^{2}+2kj+6jn\\
& +6kn-4k-4j+4n\}\\
=&6\{3k(n-k)+3(j+1)(n-j)+3n(j+k-n)+2o_{P}^{st}(r_{P}^{s}+j\\
 &-r_{P}^{t}-k-1)+2k(j-2)+2(o_{P}^{st}-2)+(n-j)+4(k+j)r_{P}^{s}\\
 &+2o_{P}^{st}k+4\}>0.
\end{align*}

Thus, $Mo(P_{2})<Mo(P_{1})$. This completes the proof.
\end{proof}
\ \notag\

By Lemma \ref{l3-1}, we can directly obtain the smallest Mostar index among $\mathcal{C}_{h}$.
\begin{lemma}\label{l3-4}\cite{clxz2021}
Let $G\in \mathcal{C}_{h}$, then
$Mo(G)\geq Mo(L_{h})=72\lfloor\frac{h}{2}\rfloor\lceil\frac{h}{2}\rceil-24\lfloor\frac{h}{2}\rfloor$, with equality iff $G\cong L_{h}$.
\end{lemma}

By Corollary \ref{c3-2}, Lemma \ref{l3-3} and Lemma \ref{l3-4}, we can obtain the smallest Mostar index among $\mathscr{P}_{h}$. The proof of Theorem \ref{t3-5} follows from the same arguments as the proof of Theorem 1.3 of \cite{dngl2020}, thus we omit the proof.
\begin{theorem}\label{t3-5}
Let $G\in \mathscr{P}_{h}$, then
$Mo(G)\geq Mo(L_{h})=72\lfloor\frac{h}{2}\rfloor\lceil\frac{h}{2}\rceil-24\lfloor\frac{h}{2}\rfloor$, with equality iff $G\cong L_{h}$.
\end{theorem}

\section{The second minimal tree-like phenylenes}
\hskip 0.6cm
Bearing in mind that $C_{L}(j,n)\in \mathcal{C}_{h,1}$ is the graph including two vertex-disjoint linear phenylene chains $L_{j}$ and $L_{n}$ as subgraphs, where $j\leq n$, and $h=j+n+1$. $C_{L}(j,k,n)\in \mathcal{C}_{h,2}$ is the graph including three vertex-disjoint linear phenylene chains $L_{j}$, $L_{k}$ and $L_{n}$ as subgraphs, where $j\leq n$, the second segment is a zigzag segment and $h=j+k+n+2$.
In the following, we give some useful results for our proofs of main theorem.

Using the cut method to $C_{L}(j,h-j-1)$ and $L_{h}$, and comparing the change of Mostar index among $C_{L}(j,h-j-1)$ and $L_{h}$, we also have
\begin{lemma}\label{l4-1}\cite{clxz2021}
Let $G= C_{L}(j,h-j-1)$, $1\leq j\leq \lfloor\frac{h-1}{2}\rfloor $, be the phenylene chain with $h$ hexagons, then $Mo(C_{L}(1,h-2)<Mo(C_{L}(2,h-3)<Mo(C_{L}(3,h-4)<\cdots <Mo(C_{L}(\lfloor\frac{h-1}{2}\rfloor,\lceil\frac{h-1}{2}\rceil)$.
\end{lemma}

\begin{figure}[ht!]
  \centering
  \scalebox{.08}[.08]{\includegraphics{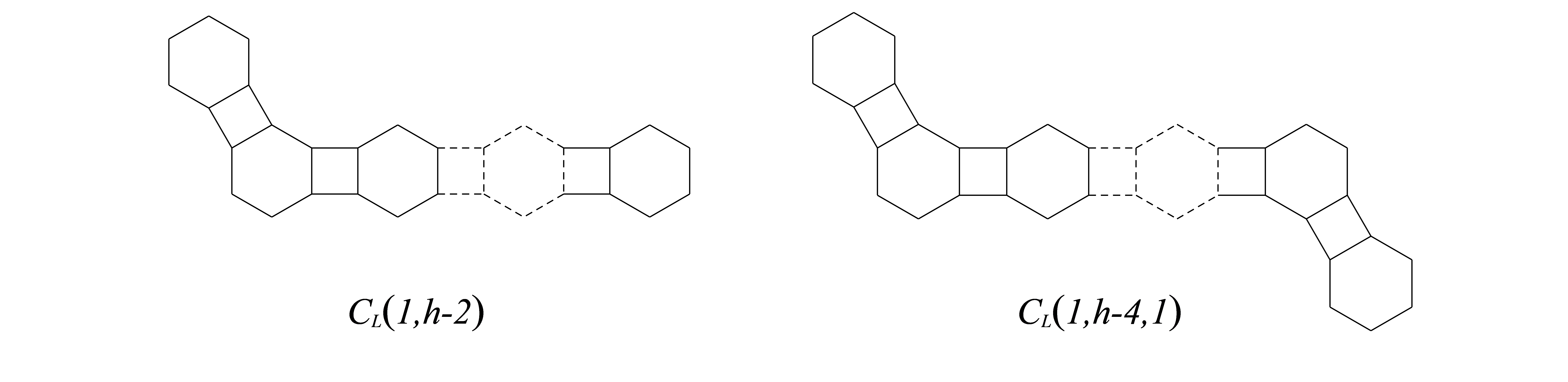}}
  \caption{The phenylenes $C_{L}(1,h-2)$, $C_{L}(1,h-4,1)$ of Lemma \ref{l4-2}.}
 \label{fig-3}
\end{figure}

By Lemma \ref{l3-1}, Lemma \ref{l4-1} and Lemma \ref{l3-4}, we can also obtain the second minimum Mostar index among phenylene chains $\mathcal{C}_h$.
\begin{lemma}\label{l4-2}\cite{clxz2021}
If $G\in \mathcal{C}_h$ and $G\not\cong L_h$, then $Mo(G)\geq Mo(C_{L}(1,h-2))=Mo(C_{L}(1,h-4,1))= 72\lfloor\frac{h}{2}\rfloor\lceil\frac{h}{2}\rceil-24\lfloor\frac{h}{2}\rfloor+24(h-1)$ with equality iff
$G\cong C_{L}(1,h-2)$ or $G\cong C_{L}(1,h-4,1)$, see Figure \ref{fig-3}.
\end{lemma}

\begin{lemma}\label{l4-3}
Given a phenylene $P_{L}(j,k,n)\in \mathscr{P}_{h,1}$ with three branches $L_{j}, L_{k}, L_{n}$ $(1\leq j\leq k\leq n)$ and $h$ $(h=j+k+n+1)$ hexagons. Then $Mo(P_{L}(j,k,n))>Mo(C_{L}(1,j+k+n-1))$ for $n\geq 2$, and $Mo(P_{L}(1,1,1))<Mo(C_{L}(1,2))$ for $n=1$.
\end{lemma}
\begin{proof}
Note that $Mo(C_{L}(1,j+k+n-1)=Mo(C_{L}(1,h-2)=72\lfloor\frac{h}{2}\rfloor\lceil\frac{h}{2}\rceil-24\lfloor\frac{h}{2}\rfloor+24(h-1)$. Then, by Lemma \ref{l2-3}, we have

\noindent {\bf Case 1}. $n\leq \lfloor\frac{h}{2}\rfloor$

\noindent {\bf Subcase 1.1}. If $h$ is even, then $n\leq j+k+1$, and we have
\begin{equation}
\begin{split}
&\quad Mo(P_{L}(j,k,n))-Mo(C_{L}(1,j+k+n-1))\\
&= 6\{-3j^{2}-3k^{2}-3n^{2}+2kj+6jn+10kn-4k-8j-1\}\\
&=  6\{3j(k-j)+3k(n-k)+3(n+1)(j+k+1-n)+k(n-j)+3(k+j)(n-4)+5k\\
& \quad  +j-4\}\\
&=  6\{3j(k-j)+3k(n-k)+3(n+1)(j+k+1-n)+k(n-j)+3(k+j)(n-3)+2k\\
& \quad -2j-4\}>0.  \nonumber
\end{split}
\end{equation}
whenever $n\geq 4$, or $(j,k,n)\in \{(1,1,3), (2,2,3), (1,3,3),(3,3,3)\}$, or $(j,k,n)=(1,2,2)$, whereas $Mo(P_{L}(1,1,1))-Mo(C_{L}(1,2))=-24<0$.

\noindent {\bf Subcase 1.2}. If $h$ is odd, then $n\leq j+k$, and we have
\begin{equation}
\begin{split}
&\quad Mo(P_{L}(j,k,n))-Mo(C_{L}(1,j+k+n-1))\\
&= 6\{-3j^{2}-3k^{2}-3n^{2}+2kj+6jn+10kn-4k-8j\}\\
&= 6\{3j(k-j)+3k(n-k)+3(n+1)(j+k-n)+k(n-j)+3(k+j)(n-4)+3n\\
& \quad +5k+j\}\\
&=  6\{3j(k-j)+3k(n-k)+3(n+1)(j+k-n)+k(n-j)+3(k+j)(n-3)\\
& \quad +2(k-j)+3n\}>0.  \nonumber
\end{split}
\end{equation}
whenever $n\geq 3$, or $(j,k,n)\in \{(2,2,2),(1,1,2)\}$.

\noindent {\bf Case 2}. $n\geq \lfloor\frac{h}{2}\rfloor+1$

$Mo(P_{L}(j,k,n))-Mo(C_{L}(1,j+k+n-1))=6\{4(j+n)(k-1)+4kj\}>0$.

The proof is completed
\end{proof}
\ \notag\

By Lemma \ref{l3-3}, Theorem \ref{t3-5}, Lemma \ref{l4-2} and Lemma \ref{l4-3}, we obtain the second minimum Mostar index of tree-like phenylenes $\mathscr{P}_{h}$.
\begin{theorem}\label{t4-4}
Let $G\in \mathscr{P}_{h}$ $(h\geq 4)$, and $G\ncong L_{h}$, then

$(1)$ If $h\geq 5$, $Mo(G)\geq Mo(C_{L}(1,h-2))=Mo(C_{L}(1,h-4,1))=72\lfloor\frac{h}{2}\rfloor\lceil\frac{h}{2}\rceil-24\lfloor\frac{h}{2}\rfloor+24(h-1)$ with equality iff
$G\cong C_{L}(1,h-2)$ or $G\cong C_{L}(1,h-4,1)$.

$(2)$ If $h=4$, $Mo(G)\geq Mo(P_{L}(1,1,1))=288$, with equality iff $G\cong P_{L}(1,1,1)$.
\end{theorem}

\section{The third minimal tree-like phenylenes}
\hskip 0.6cm
Bearing in mind $C_{L}(t_{1},t_{2},t_{3},\cdots, t_{k}, t_{k+1})$ is the phenylene chian with $h$ hexagons and exactly $k+1$ segments $S_1, S_2,\ldots,S_{k+1}$ of lengths $t_{1}+1,t_{1}+2,t_{3}+2,\ldots,t_{k}+2, t_{k+1}+1$, respectively, where $S_1$ and $S_{k+1}$ are the terminal segments, all $S_i$ $(2\leq i\leq k)$ are zagzig segments, $1\leq t_1\leq t_{k+1}$, and $\sum\limits_{i=1}^{k+1}t_{i}+k=h$.
In the following, we give the following Lemma \ref{l5-1} and Lemma \ref{l5-2}, which are important for our proofs of main theorem \ref{t5-4}.
At first, we give the third minimum Mostar index among phenylene chains.

\begin{figure}[ht!]
  \centering
  \scalebox{.077}[.077]{\includegraphics{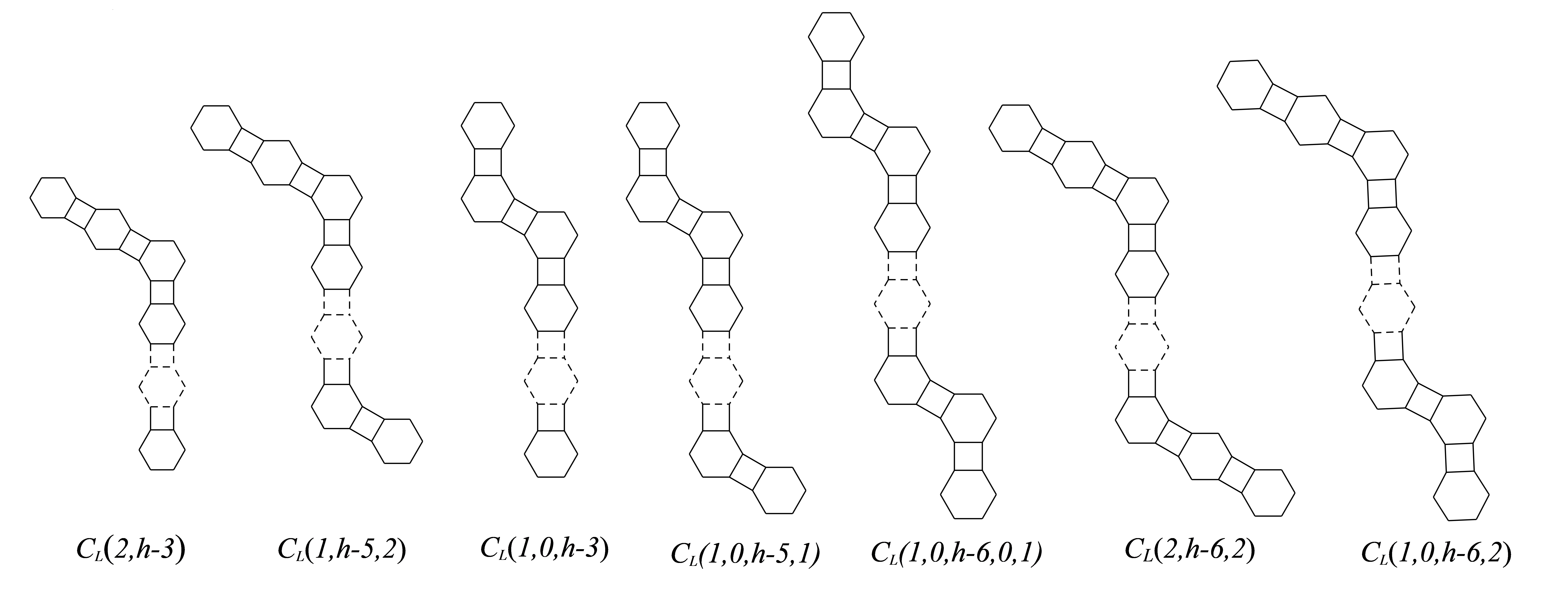}}
  \caption{Seven phenylene chains of Lemma \ref{l5-1}.}
 \label{fig-4}
\end{figure}

Let $C_{L}(j,k,n)\in \mathcal{C}_{h,2}$ with second segment is a zigzag segment. By Lemma \ref{l3-1}, Lemma \ref{l4-1}, we also have
\begin{lemma}\label{l5-1}\cite{clxz2021}
If $G\in \mathcal{C}_h$ and $G\not\cong \{L_h, C_{L}(1,h-2), C_{L}(1,h-4,1)\}$, then $Mo(G)\geq 72\lfloor\frac{h}{2}\rfloor\lceil\frac{h}{2}\rceil-24\lfloor\frac{h}{2}\rfloor+48(h-2)$ with equality iff
$G\in \{ C_{L}(2,h-3), C_{L}(1,0,h-3), C_{L}(1,h-5,2), C_{L}(2,h-6,2), C_{L}(1,0,h-5,1), C_{L}(1,0,h-6,2), C_{L}(1,0,h-6,0,1)\}$, see Figure \ref{fig-4}.
\end{lemma}

\begin{lemma}\label{l5-2}
Given a phenylene $G=P_{L}(j,k,n)\in \mathscr{P}_{h,1}$ with three branches $L_{j}, L_{k}, L_{n}$ $(1\leq j\leq k\leq n)$ and $h$ $(h=j+k+n+1)$ hexagons. Then

$(1)$ If $G\notin \{P_{L}(1,1,h-3), P_{L}(1,2,2), P_{L}(2,2,2)\}$, then $Mo(G)> Mo(C_{L}(2,h-3))$.

$(2)$ If $G\in \{P_{L}(1,1,h-3), P_{L}(2,2,2)\}$, then $Mo(G)< Mo(C_{L}(2,h-3))$.

$(3)$ If $G\cong P_{L}(1,2,2)$, then $Mo(G)= Mo(C_{L}(2,3))$.

\end{lemma}
\begin{proof}
Note that $Mo(C_{L}(2,j+k+n-2)=Mo(C_{L}(2,h-3)=72\lfloor\frac{h}{2}\rfloor\lceil\frac{h}{2}\rceil-24\lfloor\frac{h}{2}\rfloor+48(h-2)$. Then by Lemma \ref{l2-3}, we have

\noindent {\bf Case 1}. $n\leq \lfloor\frac{h}{2}\rfloor$

\noindent {\bf Subcase 1.1}. If $h$ is even, then $n\leq j+k+1$, and we have
\begin{equation}
\begin{split}
&\quad Mo(P_{L}(j,k,n))-Mo(C_{L}(2,j+k+n-2))\\
&=  6\{-3j^{2}-3k^{2}-3n^{2}+2kj+6jn+10kn-8k-12j-4n+7\}\\
&=  6\{3j(k-j)+3k(n-k)+(3n+7)(j+k+1-n)+k(n-j)+3(k+j)(n-6)\\
& \quad +3k-j\}\\
&=  6\{3j(k-j)+3k(n-k)+(3n+7)(j+k+1-n)+k(n-j)+3(k+j)(n-5)-4j\}\\
&=  6\{3j(k-j)+3k(n-k)+(3n+7)(j+k+1-n)+k(n-j)+3(k+j)(n-4)-7j\\
& \quad -3k\}>0.  \nonumber
\end{split}
\end{equation}
whenever $n\geq 6$, or $(j,k,n)\in \{(5,5,5), (3,5,5), (1,5,5),(4,4,5),(2,4,5),(3,3,5),(1,3,5),$ $(2,2,5)\}$, or $(j,k,n)\in \{(3,4,4), (1,4,4), (2,3,4),(1,2,4)\}$ , or $(j,k,n)\in \{(3,3,3), (1,3,3),$ $ (2,2,3)\}$, whereas
$Mo(P_{L}(1,2,2))-Mo(C_{L}(2,3))=0$, $Mo(P_{L}(1,1,3))-Mo(C_{L}(2,3))=-24<0$, $Mo(P_{L}(1,1,1))-Mo(C_{L}(1,2))=-48<0$.

\noindent {\bf Subcase 1.2}. If $h$ is odd, then $n\leq j+k$, and we have
\begin{equation}
\begin{split}
&\quad Mo(P_{L}(j,k,n))-Mo(C_{L}(2,j+k+n-2))\\
&= 6\{-3j^{2}-3k^{2}-3n^{2}+2kj+6jn+10kn-8k-12j-4n+8\}\\
&=  6\{3j(k-j)+3k(n-k)+(3n+4)(j+k-n)+k(n-j)+3(k+j)(n-5)+3k-j\\
& \quad +8\}\\
&=  6\{3j(k-j)+3k(n-k)+(3n+4)(j+k-n)+k(n-j)+3(k+j)(n-4)-4j+8\}\\
&=  6\{3j(k-j)+3k(n-k)+(3n+4)(j+k-n)+k(n-j)+3(k+j)(n-3)-7j-3k\\
& \quad +8\}>0.  \nonumber
\end{split}
\end{equation}
whenever $n\geq 5$, or $(j,k,n)\in \{(4,4,4), (2,4,4), (3,3,4),(1,3,4),(2,2,4)\}$, or $(j,k,n)\in \{(2,3,3), (1,2,3)\}$, whereas
$Mo(P_{L}(2,2,2))-Mo(C_{L}(2,4))=-24<0$, $Mo(P_{L}(1,1,2))-Mo(C_{L}(2,2))=-24<0$.

\noindent {\bf Case 2}. $n\geq \lfloor\frac{h}{2}\rfloor+1$
$$Mo(P_{L}(j,k,n))-Mo(C_{L}(2,j+k+n-2))=6\{4(j+n)(k-2)+4k(j-1)+8\}>0,$$
whenever $k\geq 2$, whereas $(j,k,n)=(1,1,n)$ $(n\geq 4)$ with  $Mo(P_{L}(1,1,n))-Mo(C_{L}(2,n))=24(1-n)<0$.

Thus, we have $(1)$ If $G\notin \{P_{L}(1,1,h-3), P_{L}(1,2,2), P_{L}(2,2,2)\}$, then $Mo(G)> Mo(C_{L}(2,h-3))$; $(2)$ If $G\in \{P_{L}(1,1,h-3), P_{L}(2,2,2)\}$, then $Mo(G)< Mo(C_{L}(2,h-3))$; $(3)$ If $G\cong P_{L}(1,2,2)$, then $Mo(G)= Mo(C_{L}(2,3))$.

The proof is completed
\end{proof}
\ \notag\

Comparing the Mostar indices of $P_{L}(1,2,2)$ with $P_{L}(1,1,3)$, $P_{L}(2,2,2)$ with $P_{L}(1,1,4)$. By Lemma \ref{l2-3}, we have
$Mo(P_{L}(1,1,3))-Mo(P_{L}(1,2,2))=24(30-32)=-48<0$, and $Mo(P_{L}(1,1,4))-Mo(P_{L}(2,2,2))=24(40-42)=-48<0$.  By Lemma \ref{l3-3}, Theorem \ref{t3-5}, Theorem \ref{t4-4}, Lemma \ref{l5-1} and Lemma \ref{l5-2}, we obtain the third minimum Mostar index of tree-like phenylenes .
\begin{theorem}\label{t5-4}
Let $G\in \mathscr{P}_{h}$, and $G\ncong \{L_{h}, C_{L}(1,h-2), C_{L}(1,h-4,1)\}$. Then $Mo(G)\geq Mo(P_{L}(1,1,h-3))$, with equality iff $G\cong P_{L}(1,1,h-3)$.
\end{theorem}

\section{More about (edge) Mostar indices}
\hskip 0.6cm
In this section, we investegate the correlation between boiling points (BP) of benzenoid hydrocarbons and edge Mostar indices.
The 21 benzenoid hydrocarbons were shown in Figure \ref{fig-5}.
The experimental values of boiling points of benzenoid hydrocarbons of Table 1 were taken from \cite{mila2021}.
The experimental values of  Mostar indices of $21$ benzenoid hydrocarbons of Table 1 were taken from \cite{dgli2021}. With the data of Figure \ref{table9}, scatter plots between BP and edge Mostar indices were shown in Figures \ref{fig-6}. We obtain that the correlation coefficient ($R$) between boiling points and edge Mostar indices is about 0.9647, and
$$BP= 0.9092\times Mo_{e}(G)+252.6.$$

From \cite{dgli2021}, we know that the correction coefficient (R) between boiling ponits of benzenoid hydrocarbons and Wiener index is 0.9642, Mostar index is 0.9573, the first status connectivity index is 0.9677, the second status connectivity index is 0.9165, the first eccentric connectivity index is 0.9315, the second eccentric connectivity index is 0.8263. We compare the correction coefficient of edge Mostar index with other distance-based indices, we find the edge Mostar index is also a good predictor. The boiling points and edge Mostar indices are highly correlated since the correction coefficient (R) between boiling ponits of benzenoid hydrocarbons and edge Mostar index is 0.9647.
It is worth noting that the regression model for the boilding point and edge Mostar index only applies to benzenoid hydrocarbons. We do not know whether it applies to phenylenes, which needs further study.

\begin{figure}[ht!]
  \centering
  \scalebox{.15}[.15]{\includegraphics{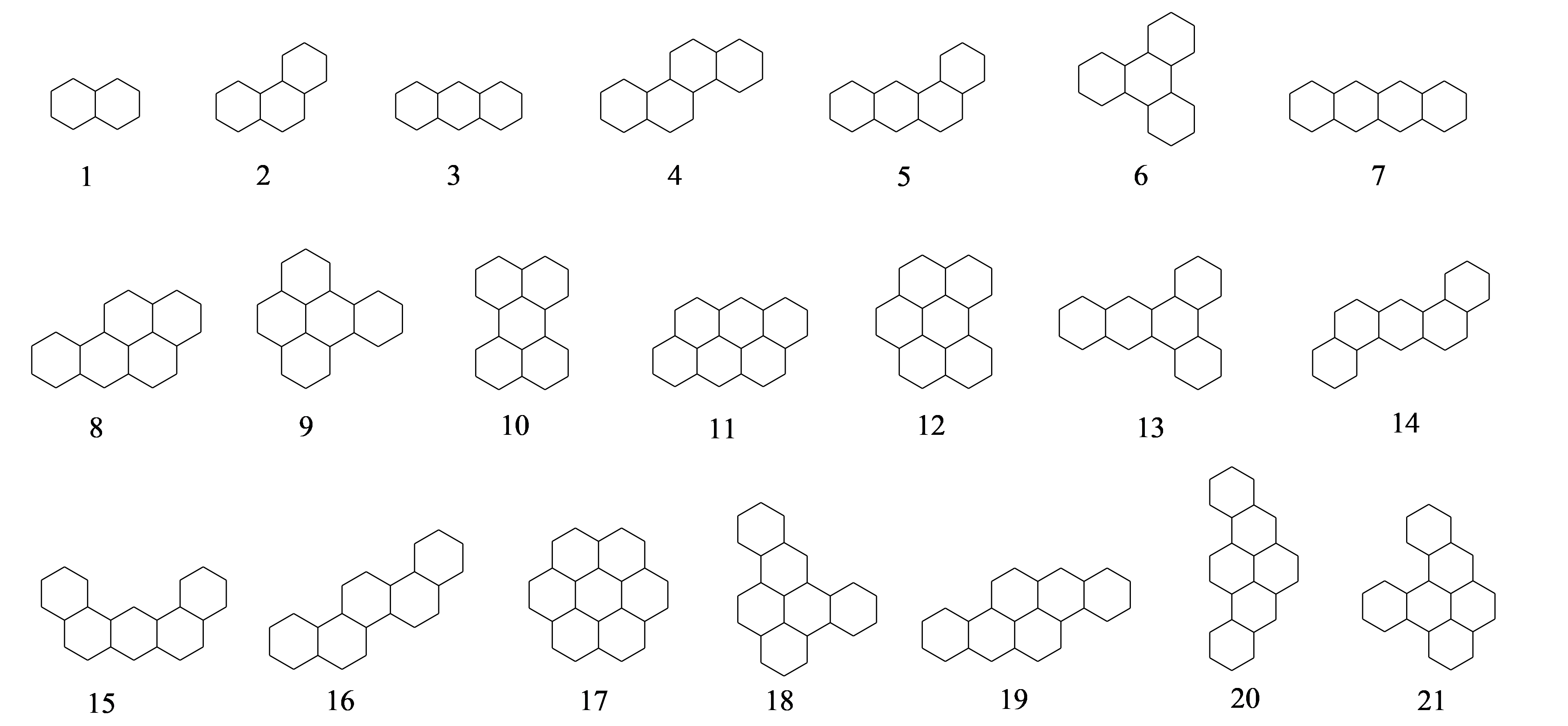}}
  \caption{21 benzenoid hydrocarbons.}
 \label{fig-5}
\end{figure}

\begin{table}[h]
	\centering
    \caption{Experimental values of BP and $Mo_{e}(G)$, $Mo(G)$ of $21$ benzenoid hydrocarbons}
    \setlength{\tabcolsep}{3mm}{
	\begin{tabular}{cccccccc}\hline
	  No.  &  $BP(^{o}C)$  & $Mo(G)$  &	$Mo_{e}(G)$ &  No.   &  $BP(^{o}C)$  & $Mo(G)$  &	$Mo_{e}(G)$     \\  \hline

	$1$    &  $218$  &    $32$  &    $40$   &  $12$   &  $542$  &    $224$  &  $300$         \\  \hline

    $2$    &  $338$  &    $88$  &	  $110$   &    $13$   &  $535$  &    $248$  &  $310$      \\ \hline

    $3$    &  $340$  &    $64$  &	  $80$   &   $14$   &  $536$  &    $232$  &	  $290$     \\ \hline

    $4$    &  $431$  &    $160$  &	  $200$   &  $15$   &  $531$  &    $264$  &	  $330$       \\ \hline

    $5$    &  $425$  &    $160$  &	  $200$  &   $16$   &  $519$  &    $256$  &	  $320$      \\ \hline

    $6$    &  $429$  &    $144$  &	  $180$  &  $17$   &  $590$  &    $252$  &	  $342$      \\ \hline

    $7$    &  $440$  &    $128$  &	  $160$  &  $18$   &  $592$  &    $302$  &	  $390$    \\ \hline

    $8$    &  $496$  &    $198$  &	  $258$  &  $19$   &  $596$  &    $300$  &	  $388$     \\ \hline

    $9$    &  $493$  &    $184$  &	  $240$  &  $20$   &  $594$  &    $300$  &	  $388$    \\ \hline

    $10$   &  $497$  &    $172$  &	  $222$  &  $21$   &  $595$  &    $322$  &	  $412$    \\ \hline

    $11$   &  $547$  &    $236$  &	  $316$       \\ \hline

	\end{tabular}}
	
	\label{table9}
\end{table}

\begin{figure}[ht!]
  \centering
  \scalebox{.06}[.06]{\includegraphics{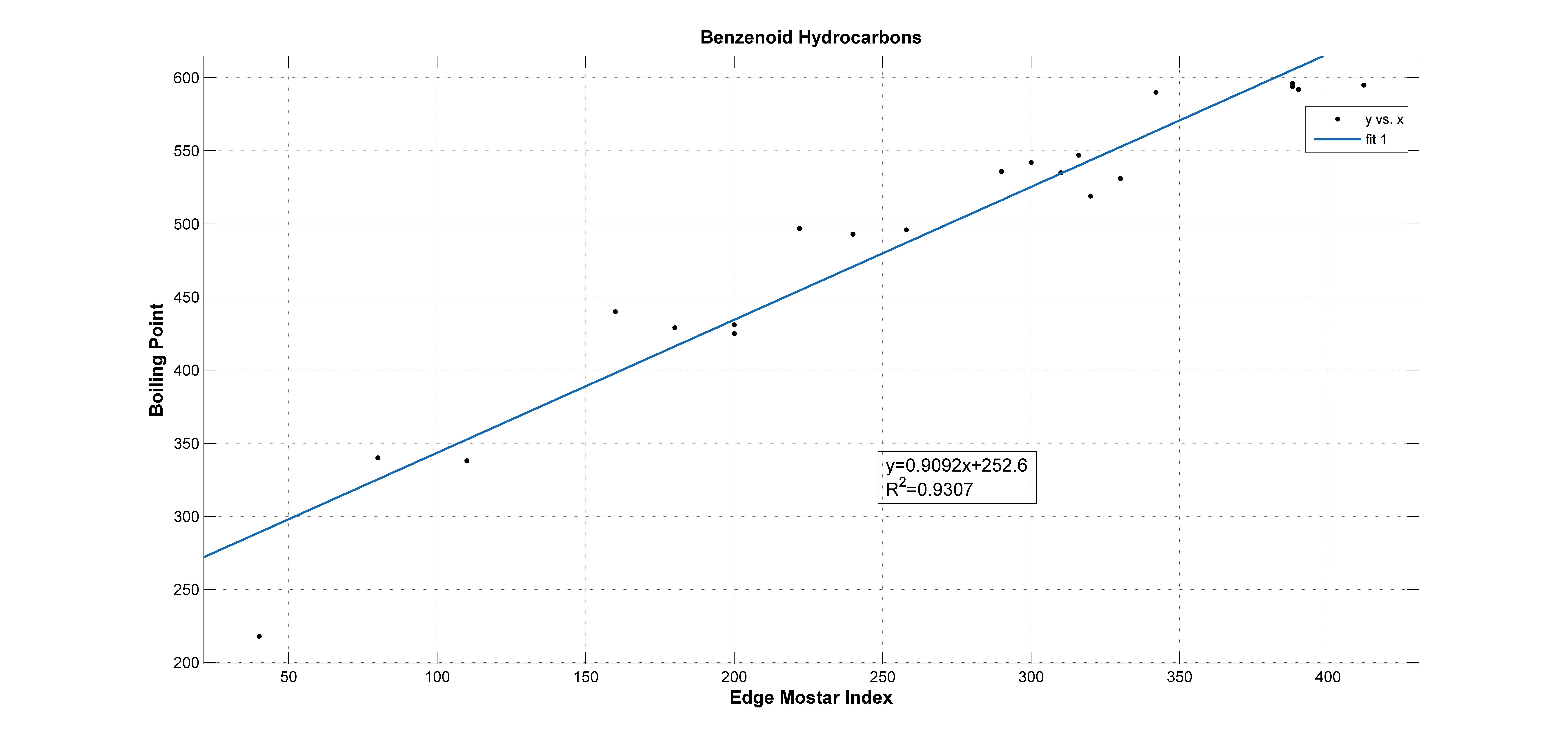}}
  \caption{Scatter plot between BP of benzenoid hydrocarbons and $Mo_{e}(G)$.}
 \label{fig-6}
\end{figure}

\section{Concluding Remarks}
\hskip 0.6cm
Among the research of quantum chemistry, computational chemistry and mathematical chemistry, the research of chemical indices is currently one of the more popular areas, as these chemical indices have proven to have a wide range of applications in QSAR, QSPR relationships for new drug discovery, molecular design, hazard estimation of compounds, numerical coding of chemical structures, database search, prediction of bioactivity, prediction of physicochemical properties of molecular materials.
In this study, we determine the first three minimum values of the Mostar index of tree-like phenylenes with a fixed number of hexagons and characterize all the tree-like phenylenes attaining these values. Quite unexpectedly, the minimum and second minimum tree-like phenylenes are all in the phenylene chains, but the third minimum tree-like phenylenes are not in the phenylene chains. The results could be of some interest to researchers working in chemical applications of graph theory.
\\

\end{document}